\documentclass{commat}

\usepackage[english]{babel}
\usepackage{amsmath,amsfonts,amsthm,amssymb}
\usepackage{graphicx}
\usepackage{cmap}
\usepackage{microtype}
\usepackage{tikz}

\newtheorem{thm}{Theorem}[section]

\newtheorem{prop}[thm]{Proposition}
\newtheorem{defn}{Definition}[section]

\title{%
    On the primitive subspace of the Lando framed graph bialgebra
    }

\author{%
    Maksim Karev
    }

\authorinfo[%
    M. Karev]{
    Guangdong-Technion Israel Institute of Technology,
China}{%
    maksim.karev@gtiit.edu.cn
    }

\abstract{%
    Lando framed graph bialgebra is generated by framed graphs modulo 4-term relations. We provide an explicit set of generators of its primitive subspace and a description of the set of relations between the generators. We also define an operation of leaf addition that endows the primitive subspace of Lando algebra with a structure of a module over the ring of polynomials in one variable and construct a 4-invariant that satisfies a simple identity with respect to the vertex-multiplication.
    }



\VOLUME{31}
\YEAR{2023}
\NUMBER{3}
\firstpage{125}
\DOI{https://doi.org/10.46298/cm.11626}

\begin{document}



\section*{Introduction}

The theory of finite-type knot invariants was first proposed by V. Vassiliev~\cite{Va}, who introduced a filtration on the space of knot invariants with finite-dimensional components. A similar filtration on invariants of plane curves was later proposed by Arnold in~\cite{Arnold}. In both cases, the graded vector spaces associated with these filtrations can be realized as subspaces of the duals to the finite-dimensional spaces of \emph{chord diagrams}~\cite{BL, Lando}, with the framing added in the case of plane curves. These associated graded spaces are known as the spaces of (framed) weight systems.

Despite significant progress, the study of the space of the finite type invariants for knots and plane curves remains incomplete. For example, even the dimensions of the spaces of weight systems are only known for the first few terms \cite{KarevHopf, Sloane1}. However, the existence of the structure of a commutative cocommutative connected bialgebra of a finite type, rich combinatorics, and its unexpected relations to other mathematical concepts (e.g. Lie algebras) make these spaces an extremely interesting object of study.

One such relation is the existence of a map from a dual of a quotient of the algebra generated by graphs on $n$ vertices to the degree $n$ grading component of the space of weight systems. This map was first described in~\cite{LandoHopf} and then extended to the framed case in ~\cite{Lando}. The above-mentioned quotient of the algebra of graphs is referred to as the Lando framed graph bialgebra, or simply the Lando bialgebra. The dimensions of the graded components of the Lando bialgebra are not known in general, see \cite{Sloane2, Sloane3}. The linear functions on the Lando bialgebra are known as framed 4-invariants. The state of the art of the study of 4-invariants can be found in~\cite{LandoKazaryan}.

The theory of weight systems is more developed than the theory of 4-invariants, with three different realizations of the space of weight systems (algebras $\mathcal A, \mathcal B$ and $\mathcal C$ of \cite{CDbook}), each with its advantages and disadvantages. For two of these realizations, the description of the corresponding primitive subspace is direct in the sense that the generators and relations are known. However, for the Lando bialgebra, the current state-of-the-art description of the primitive subspace is not direct – we can only say that it is generated by images of the generators of the graph space under the projection operator (see ~\cite{LandoHopf}).

In this note, we introduce a graph-theoretic analogue of algebra $\mathcal C$. Our proposed construction gives a more direct description of its primitive subspace in terms of generators and relations. We also introduce a graph-theoretic counterpart of the well-known operation of bubble insertion and propose a new 4-invariant.

This note is dedicated to the memory of S.V. Duzhin, who introduced the theory of finite-type knot invariants to the author.  I am grateful to B. Bychkov for valuable discussions and D. Fomichev for implementing computer codes to verify the constructions in this note. I am also grateful to Jacob Mostovoy, and the anonymous referee, whose comments allowed me to make the text clearer.

Below, $\mathbb K$ is a characteristic 0 field.

\section{Bialgebra structures on the spaces generated by isomorphism classes of framed finite simple graphs}

\begin{defn}
The framing on a finite graph is a function $f$ from the set of its vertices to $\mathbb F_2$.
\end{defn}
 The framed graphs have a naturally defined notion of framing preserving isomorphism. We will refer to the equivalence classes of framed finite simple graphs, as just framed graphs. 
 
 S.A.~Joni and G.C.~Rota~\cite{JoniRota} have proposed to endow the vector space spanned by the framed graphs with the structure of a commutative cocommutative connected bialgebra of finite type over $\mathbb K$: 
 
\begin{defn}
The graded bialgebra $G_\mathcal{JR}$ over $\mathbb K$ is spanned by framed graphs. The grading is given by the number of vertices of the graph. The product on $G_\mathcal {JR}$ is the extension by linearity of the disjoint union of graphs. The unit element is the empty graph, the counit element maps the empty graph to 1, and all the other graphs to 0. The coproduct is given by
$$\Delta_\mathcal{JR} (\Gamma) = \sum \Gamma_p\otimes \Gamma_q,$$
where the sum is taken over all possible ways to split the set of vertices of $\Gamma$ into two disjoint subsets $p$ and $q$. We denote by $\Gamma_p$ ($\Gamma_q$, respectively),  the full subgraph of $\Gamma$ generated by the set of vertices $p$ ($q$, respectively).
\end{defn} 
 
We define the following algebra closely related to $G_\mathcal{JR}$.

\begin{defn}
Let $\Gamma$ be a framed graph. A \emph{coloring} on the edges of $\Gamma$ is the function $C\colon E(\Gamma) \to \{b,r\}$. 
\end{defn}

Below we will refer to the edges with coloring $b$ (coloring $r$, respectively) as to black (respectively, red) edges\footnote{We follow \cite{L} in the choice of the colors.}.

We define the following commutative cocommutative connected bialgebra of finite type structure on the vector space spanned by the colored framed graphs:

\begin{defn}
The graded bialgebra $G_\mathcal{C}$ over $\mathbb K$ is spanned by colored framed graphs. The grading is given by the number of vertices of the graph. The product on $G_\mathcal {C}$ is the extension by the linearity of the disjoint union of graphs. The unit element is the empty graph, the counit element maps the empty graph to 1, and all other graphs to 0. The coproduct is given by
$$\Delta_\mathcal{C} (\Gamma) = \sum \Gamma_p\otimes \Gamma_q,$$
where the sum is taken over all possible ways to split the set of vertices of graph $\Gamma$ into two disjoint subsets $p$ and $q$, such that \emph{no vertex from $p$ is connected to a vertex from $q$ by a red edge}. We denote by $\Gamma_p$ ($\Gamma_q$, respectively),  the full subgraph of $\Gamma$ generated by the set of vertices $p$ ($q$, respectively), respecting the coloring of the edge.
\end{defn}

The proof that the introduced operations define a bialgebra structure is a routine verification of the axioms, which we omit.

The algebra $G_\mathcal{JR}$ admits an injective graded bialgebra map $\iota$ to the algebra $G_\mathcal{C}$. Namely, every framed graph is mapped to itself all the edges colored black. 

In this note, we use the following convention for visualizing the elements of $G_\mathcal{C}$. The color of an edge of a depicted graph corresponds to the value of the coloring function on it. Capital letters indicate the framings of the vertices. Small letters stand for subsets of vertices of $G_\mathcal{C}$; the subsets corresponding to different letters are not necessarily disjoint. A small letter written on an edge endpoint indicates that, besides the edges that are drawn explicitly, this vertex is also connected to each vertex from the subset denoted by that letter. The parts of the graphs that are omitted from the picture are assumed to be the same.

We define $I_\mathcal{C}$ to be the ideal of $G_\mathcal{C}$ spanned by all the possible elements of the form:

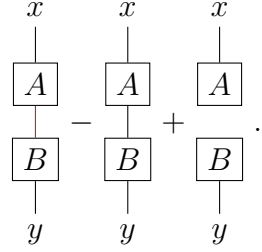
\begin{center}

\begin{tikzpicture}
		\node at (0,2) (edgesx) {$x$};
        \node at ( 0,1) (vertexu) [draw] {$A$};
        \node at (0,0) (vertexv) [draw] {$B$};
        \node at (0,-1) (edgesy) {$y$};
        
        \draw (edgesx) -- (vertexu);
      	\draw (vertexu) -- (vertexv) [color = red];
      	\draw (vertexv) -- (edgesy);
\end{tikzpicture}
\raisebox{47pt}{$-$}
\begin{tikzpicture}
		\node at (0,2) (edgesx) {$x$};
        \node at ( 0,1) (vertexu) [draw] {$A$};
        \node at (0,0) (vertexv) [draw] {$B$};
        \node at (0,-1) (edgesy) {$y$};
        
        \draw (edgesx) -- (vertexu);
      	\draw (vertexu) -- (vertexv);
      	\draw (vertexv) -- (edgesy);
\end{tikzpicture}
\raisebox{47pt}{$+$}
\begin{tikzpicture}
		\node at (0,2) (edgesx) {$x$};
        \node at ( 0,1) (vertexu) [draw] {$A$};
        \node at (0,0) (vertexv) [draw] {$B$};
        \node at (0,-1) (edgesy) {$y$};
        
        \draw (edgesx) -- (vertexu);
      	\draw (vertexv) -- (edgesy);
\end{tikzpicture}
\raisebox{47pt}{.}

\end{center}

\begin{thm}
The inclusion $\iota\colon G_\mathcal{JR} \to G_\mathcal{C}$ gives rise to a bialgebra isomorphism $\phi\colon G_\mathcal{JR} \to G_\mathcal{C}/I_\mathcal{C},$ where the bialgebra structure on $G_\mathcal{C}/I_\mathcal{C}$ is induced from $G_\mathcal{C}$.
\end{thm}

\begin{proof}

We define the map $G_\mathcal{C}$ to $G_\mathcal{JR}$ on the generators as follows. Let $k$ be the set of the red edges of graph $\Gamma$. To every red edge of the graph,, we assign one of two states: state $b$ corresponds to changing the color of the corresponding edge to black, and state $r$ corresponds to removing the edge. For a collection of states $p \in \{b;r\}^k$, 
denote by $\Gamma_p$ the graph obtained from $\Gamma$ by removing the edges for which the value of $p$ is $r$ and painting the rest of the edges black. 

Interpret this graph as an element of $G_\mathcal{JR}$. Define
$$ \psi(\Gamma) = \sum_{p \in \{b;r\}^k} (-1)^{p_r} \Gamma_p, $$
where $p_r$ means the number of edges with state $r$ in $p$.

The map $\psi$ evaluated on any element of $I_\mathcal{C}$ is 0. Moreover, its restriction on the image of $\iota$ is a two-sided inverse to $\iota$. As any element of $G_\mathcal {C}$ modulo $I_\mathcal{C}$ is equivalent to a linear combination of graphs with all the edges colored black, the isomorphism on the level of algebras follows.

The ideal $I_\mathcal{C}$ satisfies
$$ \Delta_\mathcal{C} I_\mathcal{C} \subset I_\mathcal{C} \otimes G_\mathcal{C} + G_\mathcal{C}\otimes I_\mathcal{C},$$
which implies that the bialgebra structure on the quotient $G_\mathcal{C}/I_\mathcal {C}$ is well-defined. The fact that the map $\psi$ is a coalgebra morphism is a routine check we omit.
\end{proof}

According to the Milnor-Moore theorem~\cite{MM} any commutative cocommutative graded connected bialgebra $A$ of finite type  over $\mathbb K$ with the coproduct operation $\Delta$ is isomorphic to the symmetric algebra of its primitive subspace, that is, the graded subspace of $A$ formed by the elements $p\in A$ such that 
$$\Delta (p) = p\otimes 1 + 1 \otimes p.$$ Given a bialgebra $A$, we will denote its primitive subspace by $PA$.

In particular, it follows that the restrictions of the maps $\phi$ and $\psi$ to the corresponding primitive subspaces are graded linear isomorphisms. It turns out that the primitive subspace of $G_\mathcal{C}/I_\mathcal{C}$ admits a simple description.

\begin{prop}
The primitive subspace of $G_\mathcal{C}/I_\mathcal{C}$ is generated by the connected framed graphs with all the edges colored red.
\end{prop}

\begin{proof}
Using the relations, we can represent any element of $G_\mathcal{C}/I_\mathcal{C}$ as a linear combination of classes of the graphs with all the edges colored red. On the other hand, the class of any connected framed graph with all the edges colored red is a primitive element. Also, the relations allow us to represent any element of the factor as a disjoint union of connected framed graphs colored red. The assertion follows.
\end{proof}

The isomorphism between $G_\mathcal{JR}$ and $G_\mathcal {C}/I_\mathcal{C}$ implies the following formula for the projection operator from $G_\mathcal{JR}$ to the subspace of its primitive elements.

\begin{prop}
Let $\pi_\mathcal{JR}$ is a linear endomorphism of $G_\mathcal{JR}$ defined on the generators as
$$ \pi_\mathcal{JR} (\Gamma) = \sum_{\Gamma'' \subset \Gamma' \subset \Gamma} (-1)^{e(\Gamma') - e(\Gamma'')} \Gamma''$$
where the summation goes along all possible pairs $\Gamma'' \subset \Gamma'$ of subgraphs of $\Gamma$, such that $\Gamma'$ is connected, and $\Gamma''$ a spanning subgraph of $\Gamma$.  

The map $\pi_{\mathcal{JR}}$ is a projection on the primitive subspace along the subspace of decomposable elements.
\end{prop}

\begin{proof}
The map $\pi_\mathcal{JR}$ is the composition of the following operations: the isomorphism $\phi$, the realization of the resulting element as a linear combination of the framed graphs with all the edges colored red, the projection $\pi_{\mathcal{C}}$, which is defined on the graphs with all the edges colored red as
$$\pi_{\mathcal{C}}(\Gamma) = \begin{cases} \Gamma,\quad\mbox{$\Gamma$ is connected}\\0,\quad\mbox{otherwise}\end{cases}$$
and the isomorphism $\psi$. 
\end{proof}

This formula is an alternative version of the projection formula~\cite{LandoHopf, Schmitt}. The reader is invited to compare this statement with the Remark to section 2.2 of~\cite{LandoKazaryan} ---  it describes, implicitly, the result of applying $\psi$ to a primitive element of $G_\mathcal{C}/I_\mathcal{C}$.

We would like to remark that the algebra $G_\mathcal{C}/I_\mathcal{C}$ admits one more projection to the subspace of primitives: a graph with all the edges colored red is mapped to the join of its connected components. It would be interesting to get an explicit description of its kernel.

Recall, that the \emph{Lando framed graph bialgebra $\mathcal L$} (or just the \emph{Lando bialgebra}, \cite{Lando, LandoHopf}) is defined as the quotient of $G_{\mathcal{JR}}$ by the biideal  $\mathcal F_{\mathcal {JR}}$ generated by \emph{4-elements}, usually written in the form.
$$\Gamma -\Gamma'_{uv} - (-1)^{f(v)}(\tilde \Gamma_{uv} - \tilde \Gamma'_{uv}).$$

This formula has the following meaning. Let $\Gamma$ be a graph, and let $u, v$ be two of its vertices joined by an edge. Then $\Gamma'_{uv}$ denotes the graph obtained from $\Gamma$ by erasing the edge between $u$ and $v$. The graph $\tilde \Gamma_{uv}$ is obtained from $\Gamma$ by the following operation: for every vertex $w$ different from $u,v$ and connected by an edge with $v$, the vertices $u$ and $w$ are joined by an edge in $\tilde \Gamma_{uv}$ if and only if the vertices $u$ and $w$ are not joined in $\Gamma$. The adjacencies of all other possible pairs of vertices in $\Gamma$ and $\tilde \Gamma_{uv}$ are the same. The graph $\tilde \Gamma'_{uv}$ is obtained from $\tilde \Gamma_{uv}$ by erasing the edge between $u$ and $v$. Finally, the framing of the vertex $u$ in both $\tilde \Gamma_{uv}$ and $\tilde \Gamma'_{uv}$ is set to $f(u) + f(v)$. 

Using the isomorphism between $G_\mathcal{JR}$ and $G_\mathcal{C}/I_\mathcal{C}$ any 4-element can be written as a class of the following element of $G_\mathcal{C}$:
\begin{gather} \label{fig:4el}
\begin{tikzpicture}[baseline=13pt,scale = 1]
		\node at (0,2) (edgesx) {$x$};
        \node at ( 0,1) (vertexu) [draw] {$A$};
        \node at (0,0) (vertexv) [draw] {$B$};
        \node at (0,-1) (edgesy) {$y$};        
        \draw (edgesx) -- (vertexu);
      	\draw (vertexu) -- (vertexv) [color = red];
      	\draw (vertexv) -- (edgesy);
\end{tikzpicture} -\, (-1)^B
\begin{tikzpicture}[baseline=13pt,scale = 1]
		\node at (0,2) (edgesx) {$x\vartriangle y$};
        \node at ( 0,1) (vertexu) [draw] {$A+B$};
        \node at (0,0) (vertexv) [draw] {$B$};
        \node at (0,-1) (edgesy)  {$y$};   
        \draw (edgesx) -- (vertexu);
      	\draw (vertexu) -- (vertexv) [color = red];
      	\draw (vertexv) -- (edgesy);
\end{tikzpicture}
\end{gather}
Here, the vertex $u$ has framing $A$, the vertex $v$ has framing $B$, $x$ and $y$ are subsets of vertices of the graph that are joined by edges to the corresponding vertices ($x\cap y$ can be non-empty), and $x\vartriangle y$ denotes the symmetric difference of the corresponding sets of vertices. The difference between the left-hand side and the right-hand side of the expression is the following: in the graph on the right the edges of the graph that connected elements of $x\cap y$ to the vertex of framing $A$ are erased, but the edges between the vertices forming $y\setminus x$ and the vertex of framing $A$ are added. Moreover, the framing $A$ is changed to $A+B$ (recall, that the framing takes values in $\mathbb F_2$, so we can add the values). In the following, any depicted element of $G_\mathcal{C}$ means its class modulo $I_\mathcal{C}$.

The following proposition describes the image  $\mathcal F_{\mathcal {JR}}$ under the isomorphism $\phi$.

\begin{prop}\label{prop:rels}
$\mathcal F_{\mathcal {C}} = \phi(\mathcal F_{\mathcal {JR}}) $ is the graded biideal of $G_\mathcal{C}/I_\mathcal {C}$ generated by the elements:

\begin{gather}\label{fig:4elred}
\sum_{b = b_1 \cup b_2}
\begin{tikzpicture}[baseline=23pt,scale = 1]
		\node at (0,2) (edgesx) {$a$};
        \node at ( 0,1) (vertexu) [draw] {$A$};
        \node at (0,0) (vertexv) [draw] {$B$};
        \node at (-1,0) (edgesy) {$c$};
        \node at (1, 1) (edgesz1) {$b_1$};
        \node at (1,0) (edgesz2) {$b_2$};
        \draw (edgesx) -- (vertexu) [color = red];
      	\draw (vertexu) -- (vertexv) [color = red];
      	\draw (vertexv) -- (edgesy) [color = red];
      	\draw (vertexu) -- (edgesz1) [color = red];
      	\draw (vertexv) -- (edgesz2) [color = red];
\end{tikzpicture}
-\, (-1)^{B} \sum_{c = c_1 \cup c_2} 
\begin{tikzpicture}[baseline=23pt,scale = 1]
		\node at (0,2) (edgesx) {$a$};
        \node at ( 0,1) (vertexu) [draw] {$A+B$};
        \node at (0,0) (vertexv) [draw] {$B$};
        \node at (-3/2, 0) (edgesy2) {$c_1$};
        \node at (-3/2, 1) (edgesy3) {$c_2$};
        \node at (3/2, 0) (edgesz) {$b$};
        \draw (edgesx) -- (vertexu) [color = red];
      	\draw (vertexu) -- (vertexv) [color = red];
      	\draw (vertexv) -- (edgesy2) [color = red];
      	\draw (vertexu) -- (edgesy3) [color = red];
      	\draw (vertexv) -- (edgesz) [color = red];
\end{tikzpicture}
,
\end{gather}
where the subsets $a,b,c$ of vertices of the unshown part of the graph are pairwise disjoint.
\end{prop}

\begin{proof}

Take a 4-element of the form shown in equation~\ref{fig:4el}. Denote $a = x - y$, $b = x\cap y$, $c = y-x$.

Modulo $I_\mathcal{C}$, every graph with a black edge can be presented as a sum of the same graph with the corresponding edge colored red and the same graph with the corresponding edge removed. The inclusion-exclusion principle implies that the element from the hypothesis is the following combination of the 4-elements
$$\sum_{\begin{smallmatrix}b' \subset b\\ c' \subset c\end{smallmatrix}} (-1)^{|b-b'| + |c - c'|}
\left(
\begin{tikzpicture}[baseline=-15pt,scale = 1]
		\node at (0,1) (edgesx) {$a$};
        \node at ( 0,0) (vertexu) [draw] {$A$};
        \node at (0,-1) (vertexv) [draw] {$B$};
        \node at (-1,-1) (edgesy) {$c'$};
        \node at (1, -1/2) (edgesz1) {$b'$};
        \draw (edgesx) -- (vertexu) [color = red];
      	\draw (vertexu) -- (vertexv) [color = red];
      	\draw (vertexv) -- (edgesy) ;
      	\draw (vertexu) -- (edgesz1) ;
      	\draw (vertexv) -- (edgesz1) ;
\end{tikzpicture}
-\, (-1)^{B} 
\begin{tikzpicture}[baseline=-15pt,scale =1]
		\node at (0,1) (edgesx) {$a$};
        \node at ( 0,0) (vertexu) [draw] {$A+B$};
        \node at (0,-1) (vertexv) [draw] {$B$};
        \node at (-3/2, -1/2) (edgesy2) {$c'$};
        \node at (3/2, -1) (edgesz) {$b'$};
        \draw (edgesx) -- (vertexu) [color = red];
      	\draw (vertexu) -- (vertexv) [color = red];
      	\draw (vertexv) -- (edgesy2) ;
      	\draw (vertexu) -- (edgesy2) ;
      	\draw (vertexv) -- (edgesz);
\end{tikzpicture}
\right).$$
The fact that every 4-element can be realized as a linear combination of the elements of the type under discussion is obvious.
\end{proof}

Notice that as the constructed elements form a biideal of $G_\mathcal{C}/I_\mathcal{C}$, the primitive subspace of $\mathcal N = (G_\mathcal{C}/I_\mathcal{C})/\mathcal F_{\mathcal {C}}$ admits a realization as the quotient of $P(G_\mathcal{C}/I_\mathcal{C})$ by its intersection with $\mathcal F_{\mathcal {C}}$. Clearly, $\mathcal F_\mathcal{C}$ is graded by the number of connected components, so the intersection consists precisely of all elements of $\mathcal F_{\mathcal C}$ that admit a realization as a linear combination of connected graphs. It provides an explicit description of the primitive subspace of the Lando bialgebra in terms of generators and relations.

This construction allows us to answer in part exercise 8 of Chapter 14 of~\cite{CDbook}. Namely, we claim that the space $\mathcal L$ in its realization as a quotient $\mathcal N = (G_\mathcal{C}/I_\mathcal{C})/\mathcal F_{\mathcal {C}}$ can be treated as an analogue of the algebra $\mathcal C$.  Indeed:
\begin{itemize}
\item there is a natural inclusion of $\mathcal L$ to this space which is in fact isomorphism;
\item the generators of the primitive subspace are connected graphs.
\end{itemize}
The elements of $I_\mathcal{C}$ can be naturally treated as analogues of the STU-relations.

It was shown in \cite{KarevHopf} that the bialgebra $\mathcal L$ carries the structure of a Hopf module over its subbialgebra $B\mathcal L
$ generated by the graphs with the framing identically equal to 0. Namely, the Larson-Sweedler theorem~\cite{LS} implies that $\mathcal L$ is a free $B\mathcal L$-module generated by the subbialgebra of the graphs with the framing identically equal to 1. Denote the latter subbialgebra by $W\mathcal L$.

\begin{thm}
The primitive subspace $P\mathcal L$  admits a direct sum decomposition
$$P\mathcal L = PB\mathcal L\oplus PW\mathcal L, $$
where the subspace $ PB\mathcal L$ consists of a linear combination of graphs whose framing is identically equal to 0, and $PW\mathcal L$ consists of a linear combination of graphs whose framing is identically equal to 1.
\end{thm}

The structural results of~\cite{KarevHopf} imply that the subspace $PW\mathcal L$ is isomorphic to the subspace of $P\mathcal N$ generated by the connected graphs with red edges only, such that the framing of \emph{at least} one vertex is 1.

%
%

Using the developed framework, we can give an alternative definition of the framed chromatic polynomial. It is defined as the multiplicative function obeying the following relations :
\begin{center}
\begin{tikzpicture}
		\node at (0,2) (edgesx) {$x$};
        \node at ( 0,1) (vertexu) [draw] {$A$};
        \node at (0,0) (vertexv) [draw] {$B$};
        \node at (0,-1) (edgesy) {$y$};
        
        \draw (edgesx) -- (vertexu) [color = red];
      	\draw (vertexu) -- (vertexv) [color = red];
      	\draw (vertexv) -- (edgesy) [color = red];
\end{tikzpicture}
\raisebox{47pt}{$=(-1)^{A+B + |x\cap y|} $}
\begin{tikzpicture}
        \node at ( 0,1/2) (vertexu) [draw] {$A + B + AB$};
        \node at (0,-1) (edgesy) {$x\cup y$};
        
      	\draw (vertexu) -- (edgesy) [color = red];
\end{tikzpicture}
\raisebox{47pt}{.}
\end{center}

\section{Leaf attachment}

Let us discuss a simple corollary of the relations in $\mathcal N$.

\begin{thm}
The following identity holds:
$$
\begin{tikzpicture}[baseline=-15pt,scale = 1]
		\node at (-1,0) (edgesx) {$x$};
        \node at ( 0,0) (vertexu) [draw] {$A$};
        \node at (0,-1) (vertexv) [draw] {$B$};
        \node at (-1,-1) (edgesy)  {$y$};
        \node at (1, -1/2) (edgesz1) [draw] {$0$};
        \draw (edgesx) -- (vertexu) [color = red];
      	\draw (vertexu) -- (vertexv) [color = red];
      	\draw (vertexv) -- (edgesy) [color = red];
      	\draw (vertexv) -- (edgesz1) [color = red];
\end{tikzpicture}
\quad = \quad
\begin{tikzpicture}[baseline=-15pt,scale = 1]
		\node at (-1,0) (edgesx) {$x$};
        \node at ( 0,0) (vertexu) [draw] {$A$};
        \node at (0,-1) (vertexv) [draw] {$B$};
        \node at (-1,-1) (edgesy)  {$y$};
        \node at (1, -1/2) (edgesz1) [draw] {$0$};
        \draw (edgesx) -- (vertexu) [color = red];
      	\draw (vertexu) -- (vertexv) [color = red];
      	\draw (vertexv) -- (edgesy) [color = red];
      	\draw (vertexu) -- (edgesz1) [color = red];
\end{tikzpicture}
$$

\end{thm}
\begin{proof}
The relations applied to the edge connecting the vertex of framing $C$ to the vertex of framing $B$ read:
\begin{multline*}
\begin{tikzpicture}[baseline=-15pt,scale = 1]
		\node at (-1,0) (edgesx) {$x$};
        \node at ( 0,0) (vertexu) [draw] {$A$};
        \node at (0,-1) (vertexv) [draw] {$B$};
        \node at (-1,-1) (edgesy)  {$y$};
        \node at (1, -1/2) (edgesz1) [draw] {$C$};
        \draw (edgesx) -- (vertexu) [color = red];
      	\draw (vertexu) -- (vertexv) [color = red];
      	\draw (vertexv) -- (edgesy) [color = red];
      	\draw (vertexu) -- (edgesz1) [color = red];
      	\draw (vertexv) -- (edgesz1) [color = red];
\end{tikzpicture}
\quad + \quad
\begin{tikzpicture}[baseline=-15pt,scale = 1]
		\node at (-1,0) (edgesx) {$x$};
        \node at ( 0,0) (vertexu) [draw] {$A$};
        \node at (0,-1) (vertexv) [draw] {$B$};
        \node at (-1,-1) (edgesy)  {$y$};
        \node at (1, -1/2) (edgesz1) [draw] {$C$};
        \draw (edgesx) -- (vertexu) [color = red];
      	\draw (vertexu) -- (vertexv) [color = red];
      	\draw (vertexv) -- (edgesy) [color = red];
      	\draw (vertexv) -- (edgesz1) [color = red];
\end{tikzpicture}
\quad + \quad
\begin{tikzpicture}[baseline=-15pt,scale = 1]
		\node at (-1,0) (edgesx) {$x$};
        \node at ( 0,0) (vertexu) [draw] {$A$};
        \node at (0,-1) (vertexv) [draw] {$B$};
        \node at (-1,-1) (edgesy)  {$y$};
        \node at (1, -1/2) (edgesz1) [draw] {$C$};
        \draw (edgesx) -- (vertexu) [color = red];
      	\draw (vertexv) -- (edgesy) [color = red];
      	\draw (vertexu) -- (edgesz1) [color = red];
      	\draw (vertexv) -- (edgesz1) [color = red];
\end{tikzpicture}
\\ = (-1)^C \quad
\begin{tikzpicture}[baseline=-15pt,scale = 1]
		\node at (-1,0) (edgesx) {$x$};
        \node at ( 0,0) (vertexu) [draw] {$A$};
        \node at (0,-1) (vertexv) [draw] {$B+C$};
        \node at (-1,-1) (edgesy)  {$y$};
        \node at (3/2, -1/2) (edgesz1) [draw] {$C$};
        \draw (edgesx) -- (vertexu) [color = red];
      	\draw (vertexv) -- (edgesy) [color = red];
     	\draw (vertexv) -- (edgesz1) [color = red];
      	\draw (vertexu) -- (edgesz1) [color = red];
\end{tikzpicture},
\end{multline*}
that implies
\begin{multline*}
\begin{tikzpicture}[baseline=-15pt,scale = 1]
		\node at (-1,0) (edgesx) {$x$};
        \node at ( 0,0) (vertexu) [draw] {$A$};
        \node at (0,-1) (vertexv) [draw] {$B$};
        \node at (-1,-1) (edgesy)  {$y$};
        \node at (1, -1/2) (edgesz1) [draw] {$C$};
        \draw (edgesx) -- (vertexu) [color = red];
      	\draw (vertexu) -- (vertexv) [color = red];
      	\draw (vertexv) -- (edgesy) [color = red];
      	\draw (vertexv) -- (edgesz1) [color = red];
\end{tikzpicture}
\quad - \quad
\begin{tikzpicture}[baseline=-15pt,scale = 1]
		\node at (-1,0) (edgesx) {$x$};
        \node at ( 0,0) (vertexu) [draw] {$A$};
        \node at (0,-1) (vertexv) [draw] {$B$};
        \node at (-1,-1) (edgesy)  {$y$};
        \node at (1, -1/2) (edgesz1) [draw] {$C$};
        \draw (edgesx) -- (vertexu) [color = red];
      	\draw (vertexu) -- (vertexv) [color = red];
      	\draw (vertexv) -- (edgesy) [color = red];
     	\draw (vertexu) -- (edgesz1) [color = red];
\end{tikzpicture}
\\ = (-1)^C \quad \left(
\begin{tikzpicture}[baseline=-15pt,scale = 1]
		\node at (-1,0) (edgesx) {$x$};
        \node at ( 0,0) (vertexu) [draw] {$A$};
        \node at (0,-1) (vertexv) [draw] {$B+C$};
        \node at (-1,-1) (edgesy)  {$y$};
        \node at (3/2, -1/2) (edgesz1) [draw] {$C$};
        \draw (edgesx) -- (vertexu) [color = red];
      	\draw (vertexv) -- (edgesy) [color = red];
     	\draw (vertexv) -- (edgesz1) [color = red];
      	\draw (vertexu) -- (edgesz1) [color = red];
\end{tikzpicture}
\quad - \quad
\begin{tikzpicture}[baseline=-15pt,scale = 1]
		\node at (-1,0) (edgesx) {$x$};
        \node at ( 0,0) (vertexu) [draw] {$A + C$};
        \node at (0,-1) (vertexv) [draw] {$B$};
        \node at (-1,-1) (edgesy)  {$y$};
        \node at (3/2, -1/2) (edgesz1) [draw] {$C$};
        \draw (edgesx) -- (vertexu) [color = red];
      	\draw (vertexv) -- (edgesy) [color = red];
     	\draw (vertexv) -- (edgesz1) [color = red];
      	\draw (vertexu) -- (edgesz1) [color = red];
\end{tikzpicture}
\right).
\end{multline*}
For $C = 0$ the right-hand part vanishes identically.
\end{proof}

Recall, that the \emph{forest algebra} $\mathcal T$~\cite{CDL3} is defined as the subalgebra of $\mathcal L$ generated by trees with all the vertices having framing 0. Clearly, its image under the isomorphism $\phi$ is the subalgebra of $\mathcal N$ also generated by trees with all the vertices framing 0 and all the edges colored red. The previous proposition trivially implies the structural result of~\cite{CDL3}:

\begin{thm}
The forest algebra $\mathcal T$ is a subbialgebra of $\mathcal N$ with the dimension one primitive subspace in every grading.
\end{thm}

Indeed, it says that any tree can be obtained by a sequence of leaf attachments to the graph on a single vertex. The resulting tree does not depend on the choice of the vertices which we attach the leaves at each step of the construction.

Also, we deduce the following:

\begin{thm}
There is a well-defined action of $\mathcal T$ on $P\mathcal N$ defined on the generators as follows. For a tree $T$ with the framing of all the vertices identically equal to 0 and a connected graph $\Gamma$, choose a vertex $v$ of $T$ and a vertex $w$ of $\Gamma$, and join the chosen vertices $v$ and $w$ by an edge.
\end{thm}

In particular, we see, that the subspaces $PW\mathcal L$ have at least one non-trivial generator in every grading component: it can be obtained by the action of a tree on $n$ vertices on a single vertex graph with the framing of the vertex 1. For every natural $n$ the obtained element is non-zero, as the framed chromatic polynomial~\cite{KarevHopf} takes on it a non-zero value.

The leaf attachment operation is an intersection graph counterpart of the operation of bubble insertion~\cite{CDbook}. P. Vogel in~\cite{Vo} provides a construction of a non-trivial element of the kernel of the operation of bubble insertion. It would be interesting to know if the described action of $\mathcal T$ on $P \mathcal N$ is free.

\section{A 4-invariant related to the number of 3-colorings of a graph.}

As the elements of the Lando bialgebra can be either expressed in terms of graphs whose all edges are black or in terms of graphs whose all edges are red, the same invariant of graphs may produce two different 4-invariants.

For instance, it is known that the values of the chromatic polynomial (that is, the numbers of $k$-colorings of the vertices of a graph) are 4-invariants of graphs \cite{CDL3, LandoKazaryan}. On the other hand, the number of 3-colorings of a graph may be used to produce a new 4-invariant as follows.

For a generating element $\Gamma$ of $\mathcal N$ represented by a graph with \emph{red edges only}, define $\mathcal W(\Gamma)$ to be equal the number of proper 3-colorings of $\Gamma$ multiplied by $2^{-\chi(\Gamma)}(-1)^f$, where $\chi(\Gamma)$ is the Euler characteristics of $\Gamma$, and $f$ is the sum of framings of all the vertices of $\Gamma$.

\begin{thm}
The function $\mathcal W$ extends by linearity to a 4-invariant.
\end{thm}

\begin{proof}
We have to check that the linear extension of $\mathcal W$ to the ideal $\mathcal F_{\mathcal{CC}}$ vanishes identically.

The map that multiplies a framed graph by $(-1)^f$ and sets the framing of all the vertices to 0 is a well-defined map on the quotient modulo the 4-elements. so, for simplicity, we assume, that from now on all the vertices have framing 0.

The elements~(\ref{fig:4elred}) can be generated as follows. The initial data is the tuple $(\Gamma, v, a, b)$, where $\Gamma$ is a graph, $v$ -- a vertex of $\Gamma$, and $a,b$ are two disjoint subsets of vertices of $\Gamma$ such that no vertex of $a\cup b$ is adjacent to $v$. 

Attach a leaf $u$ to $v$. Now, every vertex in $a\cup b$ can have 3 possible states, which we call state $U$, state $V$ and state $UV$.

For the collection of states $S\colon a\cup b \to \{U,V,UV\}$ form a new graph $\Gamma_S$ as follows: connect by edges all the vertices of the state $U$ with the vertex $u$, all the vertices of the state $V$ to the vertex $v$, and all the vertices of the state $UV$ to both the vertices $u$ and $v$. Now the elements~\ref{fig:4elred} are the linear combinations
\begin{gather}\label{eq:colors}
 \sum_{\{S\colon a\cup b \to \{U,V,UV\}\,|\, S(b) = \{U\}\}} \Gamma_S - \sum_{\{S'\colon a\cup b \to \{U,V,UV\}\,|\, S'(a) = \{U\}\}} \Gamma_{S'}.
\end{gather}
Now consider proper colorings of the vertices of the graphs $\Gamma_S$ and $\Gamma_{S'}$ in three colors which we denote $E, F,H$. Without loss of generality suppose, the vertex $u$ is colored  $E$, and the vertex $v$ is colored $F$. For any vertex $w\in a$ we have
\begin{itemize}
\item if $w$ is in the state $UV$, then it can only be colored $H$,
\item if $w$ is in the state $U$, it can be either colored $F$ or $H$,
\item if $w$ is in the state $V$, it can be either colored $E$ or $H$.
\end{itemize}

Notice, that the number of edges of a graph with a vertex $w$  in the state $UV$ is one more than those of graphs with the corresponding vertex in the states $U$ or $V$. Due to the factor $(-2)^{-\chi(\Gamma)}(-1)^f$ in the definition of $\mathcal W$, with all the states of elements $a\cup b - \{w\}$ fixed, the colorings of $w$ in the state $UV$ cancel out the colorings with $w$ colored $H$ being in the state $U$ and $w$ colored $H$ being in the state $V$.

It means that having fixed the colorings of $u$ and $v$ the sum~(\ref{eq:colors}) reduces to verification of the identity
$$
 \sum_{\{S\colon a\cup b \to \{U,V\}\,|\, S(b) = \{U\}\}} \mathcal W'_{a}(\Gamma_S) = \sum_{\{S'\colon a\cup b \to \{U,V\}\,|\, S'(a) = \{U\}\}} \mathcal W'_b(\Gamma_{S'}) ,
$$
where $W'_{a}(\Gamma_S)$ for $S\colon a\cup b \to \{U,V\}$ with $S(b) = \{U\}$ is the number of proper colorings of the graph $\Gamma_S$, with vertex $u$ colored $E$, vertex $v$ colored $F$,  the vertex $w \in a$ colored $F$ if it is in the state $U$, and colored $E$, if it is in the state $V$. Notice, that due to the structure of $\Gamma_S$, the vertices of $b$ can only be colored $F$ or $H$.
The definition of $W'_b$ is essentially the same but with the roles of $a$ and $b$ interchanged.

But the identity to verify holds true.  Namely, the sets of colorings that contribute to left-hand side are in the bijection with the set of coloring that contributes to the right-hand side: keep the coloring of $u$, $v$, and all the vertices colored $F$, but for all other vertices swap the colors $E$ and $H$. After the colors swap, the colorings of the vertices of $b$ uniquely define their states.

\end{proof}

We mention one evident property of the 4-invariant $\mathcal W$. Namely, given two graphs $\Gamma$ and $\Gamma'$, a vertex $u$ of $\Gamma$ and a vertex $v$ of $\Gamma'$ denote $\Gamma \nabla_{uv} \Gamma'$ the result of identification of vertices $u$ and $v$. Clearly
$$ \mathcal W (\Gamma \nabla_{uv} \Gamma') = -\frac 23 \mathcal W(\Gamma) \mathcal W(\Gamma').$$ In particular, the operation of the leaf attachment multiplies the value of $\mathcal W$ by 2. It is known \cite{Vai} that  $\mathfrak {sl}_2$-weight system has a similar multiplicativity property with respect to the vertex-multiplication.

As a conclusion, we notice that for $n=4$ the grading component $P\mathcal N_4$ has at least 4 linearly independent elements represented by graphs with all the edges colored red with supports given by a chain $P_4$ and a cycle $C_4$ and various framings of their vertices. They differ by the values of the framed chromatic polynomial and invariant $\mathcal W$ on them. The existence of the action of the tree algebra on $P\mathcal N$ implies, that for any $n\ge 4$ the dimension $\dim P\mathcal N_n$ is greater or equal to 4.

{\small\bibliography{commat}}
\EditInfo{July 21, 2023}{January 5, 2024}{Jacob Mostovoy, Sergei Chmutov.}
\end{document}